\documentclass[reqno,11pt]{amsart}
\title[Spectra of Weighted Composition Operators]{Spectra of Some Weighted Composition Operators on $H^{2}$}
\author{Carl C. Cowen, Eungil Ko, Derek Thompson, and Feng Tian}%
\address{IUPUI (Indiana University -- Purdue University, Indianapolis), Indianapolis, Indiana 46202-3216}%
\email{ccowen@math.iupui.edu}%
\address{Ewha Womans University, Seoul 120-750, S. Korea}
\email{eiko@ewha.ac.kr}
\address{Trine University, Angola, Indiana 46703}
\email{theycallmedt@gmail.com}
\email{tianf@trine.edu}
\date{16 May 2014}%
\thanks{The second author was supported by Basic Science Research Program
through the National Research Foundation of Korea (NRF) grant funded
by the Ministry of Education, Science and Technology (2012R1A2A2A02008590).}%
\newtheorem{thm}{Theorem}
 \newtheorem{lemma}[thm]{Lemma}
 
 \newtheorem{cor}[thm]{Corollary}
 \newtheorem{example}[thm]{Example}
 \newcommand{\dfn}{\noindent {\bf Definition} }
  \renewcommand{\Re}{{\rm Re}}

   \newcommand{\ph}{\ensuremath{\varphi}}

   \newcommand{\Ht}{\ensuremath{H^2}}
   \newcommand{\HtD}{\ensuremath{H^2(\D)}}
   \newcommand{\Hi}{\ensuremath{H^\infty}}
   
   \newcommand{\D}{\ensuremath{\mathbb{D}}}
   
   \newcommand{\Co}{\ensuremath{\mathbb{C}}}
   \newcommand{\C}{\ensuremath{C_{\varphi}}}
   
   \newcommand{\T}{\ensuremath{T_\psi}}
   
  \newcommand{\W}{\ensuremath{W_{\psi,\varphi}}}
  \newcommand{\Ws}{\ensuremath{W_{\psi,\varphi}^{\textstyle\ast}}}
    \newfont{\caps}{cmcsc10}  
  \newfont{\jour}{cmti10}  
  \newcommand{\vypp}[4]{ {\bf #1}(#2), #3--#4.}

\newcommand{\tams}[4]{{\jour Trans.~Amer.~Math.~Soc.}\vypp{#1}{#2}{#3}{#4}}
 \newcommand{\canj}[4]{{\jour Canadian~J.~Math.}\vypp{#1}{#2}{#3}{#4}}

  \newcommand{\ill}[4]{{\jour Illinois~J.~Math.}\vypp{#1}{#2}{#3}{#4}}
  
  \newcommand{\inteq}[4]{{\jour Integral Equations Operator Theory}\vypp{#1}{#2}{#3}{#4}}
  \newcommand{\jfa}[4]{{\jour J. Functional Analysis}\vypp{#1}{#2}{#3}{#4}}
  \newcommand{\jmaa}[4]{{\jour J. Math. Anal. App.}\vypp{#1}{#2}{#3}{#4}}

\begin{document}

\begin{abstract}
We completely characterize the spectrum of a weighted composition
operator \W\ on \HtD\ when \ph\ has Denjoy-Wolff point $a$ with $0<|\ph'(a)|< 1$,
the iterates, $\ph_n$, converge uniformly to $a$, and $\psi$ is in \Hi\ and continuous
at $a$. We also give bounds and some computations 
when $|a|=1$ and $\ph'(a)=1$ and, in addition, show that these symbols include
all linear fractional \ph\ that are hyperbolic and parabolic non-automorphisms.
Finally, we use these results to eliminate possible weights  $\psi$
so that  \W\ is seminormal. 
\end{abstract}
\maketitle

\section{Introduction}

If $\psi$ is in \Hi\ and \ph\ is analytic map of the unit
disk into itself, the \emph{weighted composition operator on
\Ht\ with symbols $\psi$ and \ph} is the operator \W,
where $\T$ is the analytic Toeplitz operator given by $\T(h)=\psi h$
for $h$ in \Ht,  \C\ is the composition operator on \Ht\
given by $\C(h)=h\circ\varphi$. Clearly, if $\psi$ is bounded on
the disk, then \W\ is bounded on $\Ht$ and $\|\W\|\leq \|\psi\|_{\infty}\|\C\|$.
Although it will have little impact on our work, it is not necessary
for $\psi$ to be bounded for \W\ to be bounded. To avoid trivialities
and special cases, we will assume $\psi$ is not identically zero
and $\ph$ is not a constant mapping.

Weighted composition operators have been studied occasionally over
the past few decades, but have usually arisen in answering other questions
related to operators on spaces of analytic functions, such as questions
about multiplication operators or composition operators. For example,
Forelli~\cite{forelli} showed that the only isometries of $H^{p}$
for $p\neq2$ are weighted composition operators and that the isometries
for $H^{p}$ with $p\neq2$ have analogues that are isometries of
$H^{2}$ (but there are also many other isometries of $H^{2}$). Weighted
composition operators also arise in the description of commutants
of analytic Toeplitz operators (see for example~\cite{commut,commut2}
and in the adjoints of composition operators (see for example~\cite{Co88,book,adj}).

Recently, work has begun on studying the spectrum of weighted composition
operators on $H^{2}$ more carefully. Gunatillake~\cite{gunatTh}
characterized the spectrum when \ph\ has an interior fixed point
and \W\ is compact. The first two authors~\cite{wcomp} characterized
the spectrum when \W\ is a self-adjoint operator. Bourdon and Narayan
extended their work~\cite{BN} to characterize the
spectrum when \W\ is unitary and when \W\ is normal with interior
fixed point. Gunatillake~\cite{gunatInv} defined invertible
weighted composition operators on $H^{2}$ and identified their spectrum.
Very recently, Hyv\"arinen, Lindstr\"om, Nieminen, and Saukko~\cite{HLNS}
extended his work to when \ph\ is an automorphism but \W\ is not
necessarily invertible.

Our work finds the spectra of \W\ with relatively weak conditions
on $\psi$, but a rather strong one on \ph, which is that the iterates of  \ph\
converge uniformly on \emph{all} of \D\ to the Denjoy-Wolff point $a$,
rather than just on compact subsets of  \D. In Section 2, we identify situations when \ph\ 
satisfies this uniformity condition on the  convergence of its iterates,
and show that this class of symbols is non-trivial. In Section
3, we give general bounds for $\sigma(\W)$ that define the spectrum
when $\sigma(\C)$ is given by the closure of $\sigma_{p}(\C)$.
In Section 4, we are much more specific about $\sigma_{p}(\W)$ when
$\ph'(a)<1$ and give some examples. In Section 5, we eliminate some
possibilities where \W\ could be seminormal. Finally, we suggest
further areas of study in Section 6.

\section{When are the iterates of \ph\ uniformly convergent?}

To accomplish the work of this paper, we make a rather strong assumption
that \ph\ converges uniformly on all of \D\ to the Denjoy-Wolff
point $a$. Our work in this section will further explain when this
phenomenon occurs. To facilitate reference to the property of uniform convergence
of the iterates of \ph, we make the following definition. \\[-2ex] 

\dfn We say \emph{UCI holds for \ph} or \emph{\ph\ satisfies UCI} if \ph\ is
an analytic map of the unit disk \D\ into itself with Denjoy-Wolff point $a$
and the iterates $\ph_n$ of \ph\ converge uniformly, on all of \D, to $a$.\\[-2ex]

We begin by showing that this condition is not particularly
helpful when the Denjoy-Wolff point $a$ belongs to \D.
\begin{thm}
\label{DbarinD} Suppose $\ph:\D\rightarrow\D$ is analytic and continuous
on $\partial\D$. If the Denjoy-Wolff point $a$ of \ph\ is in
\D, then $\ph_{n}\rightarrow a$ uniformly if and only if there
is $N>0$ such that $\ph_{N}(\overline{\D})\subseteq\D$.\end{thm}
\begin{proof}
Suppose there is $N>0$ such that $\ph_{N}(\overline{\D})\subseteq\D$.
Since $\ph_{n}$ always converges uniformly on compact subsets of
\D\ to $a$ by the Denjoy-Wolff Theorem~\cite{book}
and $\ph_{N}(\overline{\D})$ is a compact subset of \D, we have
that $\ph_{n}\rightarrow a$ uniformly on \D.

To prove the other direction, let $M$ be the minimum distance between
$a$ and the unit circle. Since $\ph_{n}\rightarrow a$ uniformly
on \D, for $\epsilon=M/2$, there exists $N>0$ such that $|\ph_{N}(z)-a|<\epsilon,\forall z\in\D$.
Suppose $\ph_{N}(b_{1})=b_{2},|b_{1}|=|b_{2}|=1$. Then for our given
$\epsilon$, since $\ph_{N}$ is continuous on the unit circle, there
exists $\delta>0$ so that $|b_{1}-z|<\delta\Rightarrow|b_{2}-\ph_{N}(z)|<\epsilon$.
However, for $z$ such that $|b_{1}-z|<\delta$, 
\begin{align*}
M\leq\left|b_{2}-a\right| & =\left|b_{2}-\ph_{N}(z)+\ph_{N}(z)-a\right|\\
 & \leq\left|b_{2}-\ph_{N}(z)|+|\ph_{N}(z)-a\right|\\
 & <2\epsilon=M
\end{align*}
which is a contradiction, so $\ph_{N}(\overline{\D})\subseteq\D$. 
\end{proof}
The following corollary shows that this is of interest.
\begin{cor}
Suppose $\ph:\D\rightarrow\D$ is analytic and continuous on $\partial\D$.
If the Denjoy-Wolff point $a$ of \ph\ is in \D\ and $\ph_{n}\rightarrow a$
uniformly, then \C\ is power-compact. Furthermore, any associated
weighted composition operator \W\ with $\psi\in \Hi$ is
power-compact. \end{cor}
\begin{proof}
Since $\ph(\overline{\D})\subseteq\D$ is a sufficient condition
for \C\ to be compact~\cite{book}, we see that
by Theorem~\ref{DbarinD} $C_{\varphi_N}=\C^N$ is compact
for some $N>0$ and \C\ is power-compact. Since compact operators
are an ideal in $\mathcal{B}(H^{2})$, $(\W)^{N}=T_{\zeta}C_{\varphi_N}$
is compact, where $\zeta=\psi(\psi\circ\ph)...(\psi\circ\ph_{N-1})$. 
\end{proof}
Since Gunatillake~\cite{gunatTh} and others have already characterized
the spectrum of compact weighted composition operators (and therefore power
compact weighted composition operators) when \ph\ has an interior fixed
point, we will instead turn our work to when the Denjoy-Wolff point
is on $\partial\D$, although our results will include the interior
fixed point case. Next, we indicate some conditions on \ph\ when
the Denjoy-Wolff point is on $\partial\D$, and give some examples.
\begin{thm}
If $\ph:\D\rightarrow\D$ is analytic in \D\ and continuous on
$\partial \D$, has Denjoy-Wolff point $a$ with $|a|=1$ and $\ph_{n}\rightarrow a$
uniformly, then $a$ is the only fixed point of \ph\ in the closed
unit disk.\end{thm}
\begin{proof}
Suppose $\ph(b)=b$, $b\neq a$. Since the Denjoy-Wolff point is on the
boundary, we must have $|b|=1$, or else $b$ would be the Denjoy-Wolff
point. Since $\ph_{n}\rightarrow a$ uniformly, given $\epsilon>0$,
there is an $N$ such that $|\ph_{N}(z)-a|<\epsilon,\forall z\in\D$.
Note that $\ph_{N}$ is continuous at $b$. For the same $\epsilon$,
there is $\delta$ such that $|b-z|<\delta\Rightarrow|\ph_{N}(b)-\ph_{N}(z)|=|b-\ph_{N}(z)|<\epsilon$.
Let $z$ be such that $|b-z|<\delta$. Then 
\[
\left|b-a\right|=\left|b-\ph_{N}(z)+\ph_{N}(z)-a\right|\leq\left|b-\ph_{N}(z)\right|+\left|\ph_{N}(z)-a\right|<2\epsilon
\]
However, if we take $\epsilon<|b-a|/2$, we have a contradiction. 
\end{proof}
Although our work so far indicates that the class of weighted composition operators
where \ph\ satisfies UCI  may be small, we now give sufficient
conditions for \ph\ to satisfy UCI and follow with some
examples. Much of the following proof is owed to~\cite{BMS}.
\begin{thm}
\label{julia} Suppose $\ph:\D\rightarrow\D$ is analytic in \D\
and continuous on $\partial \D$ and has Denjoy-Wolff point $a$
with $|a|=1,\ph'(a)<1$. If $\ph_{N}(\overline{\D})\subseteq\D\cup\{a\}$
for some $N$, then $\ph_{n}\rightarrow a$ uniformly in \D.\end{thm}
\begin{proof}
Without loss of generality, we will assume $\ph(\overline{\D})\subseteq\D\cup\{a\}$.
Since $\ph(\overline{\D})\subseteq\D\cup\{a\}$ and $\ph(\overline{\D})$
is connected, it fits within the disk 
\[
H(a,\lambda):=\left\{ z\in\Co:\left|a-z\right|^{2}\leq\lambda\left(1-\left|z\right|^{2}\right)\right\} 
\]
for some fixed $\lambda>0$. Disks of this type are Euclidean subdiscs
of \D\ centered at $a/(1+\lambda)$ with radius $\lambda/(1+\lambda)$,
and are tangent to the unit circle at $a$. Julia's Lemma~\cite{book}
shows that $\ph(H(a,\lambda))\subseteq H(a,\ph'(a)\lambda)$. Applying
\ph\ iteratively, we see that for any $z$ in this set, we have
\[
|a-\ph_{n}(z)|^{2}\leq\ph'(a)^{n}\lambda(1-|\ph_{n}(z)|^{2})
\]
and therefore 
\[
|a-\ph_{n}(z)|\leq\sqrt{\lambda}\ph'(a)^{n/2}(1-|\ph_{n}(z)|)\leq\sqrt{\lambda}\ph'(a)^{n/2}
\]

Thus, for any $\epsilon>0$, there is $N>0$ such that for $n>N$,
$\sqrt{\lambda}\ph'(a)^{n/2}<\epsilon$ (since $\ph'(a)<1$). Then
$|\ph_{n}(z)-a|\leq\sqrt{\lambda}\ph'(a)^{n/2}<\epsilon$ for $n>N$. 
\end{proof}
Although this does not completely characterize UCI holding for \ph\
when $|a|=1$ and $\ph'(a)<1$, we see from this sufficient condition
that this class includes, at least, \ph\ that are linear fractional
non-automorphisms, such as $\ph(z)=\frac{1}{2}z+\frac{1}{2}$. When
$\ph'(a)=1$, the situation is even more delicate because the conditions
above are not sufficient, as can be seen when \ph\ is a parabolic
automorphism. However, if \ph\ is a linear fractional non-automorphism
with $\ph'(a)=1$, we see that \ph\ actually satisfies UCI:
\begin{example}
Let \ph\ be a linear fractional map, not an automorphism, with
Denjoy Wolff point $a$ such that $|a|=1$ and $\ph'(a)=1$. Without loss of generality,
assume $a=1$. Such symbols form a semigroup $\ph_{t}(z)=\frac{t+(2-t)z}{(2+t)-tz}$.
Then we have
\begin{align*}
\left|\ph_{t}(z)-1\right| & =\left|\frac{t+(2-t)z}{(2+t)-tz}-\frac{(2+t)-tz}{(2+t)-tz}\right|\\
 & =\left|\frac{2z-2}{(2+t)-tz}\right|\\
 & =\left|\frac{2(z-1)}{2+t(1-z)}\right|\\
 & =\left|\frac{2(1-z)}{2+t(1-z)}\right|\\
 & =\left|\frac{2}{\frac{2}{1-z}+t}\right|\\
 & \leq\left|\frac{2}{t}\right|=\frac{2}{t}\rightarrow0,\;\mbox{as }t\rightarrow\infty
\end{align*}
since $\Re\left\{ \frac{2}{1-z}\right\} >1$ for $z\in\D$. Thus if
\ph\ is a linear fractional non-automorphism with Denjoy-Wolff
point $a$ and $\ph'(a)=1$, then $\ph_{n}\rightarrow a$ uniformly
in \D. 
\end{example}
Now we see that UCI holding for \ph\ can arise when $\ph'(a)<1$ or
$\ph'(a)=1$. Next, we show general bounds for the spectra of a weighted
composition operator with UCI holding for the compositional symbol,
and later we discuss the differences more specifically between the
two cases.

\section{Spectral bounds for \W}

Throughout the remainder of the paper, we will assume that $\psi$ is in \Hi,
continuous at the Denjoy-Wolff point $a$ of $\ph$, and that $\psi(a)\neq0$.

We now offer some lemmas which will give us an inequality between
the spectra of \W\ and $\psi(a)\C$.
\begin{lemma}
\label{AB} If $A$ and $B$ are bounded operators on a Hilbert space
$\mathcal{H}$, then:
\begin{enumerate}
\item \label{enu:AB-1}If $ABv=\lambda v$ and $Bv\neq0$, then $Bv$ is
an eigenvector for $BA$ with eigenvalue $\lambda$.
\item \label{enu:AB-2}$\sigma(AB)\cup\{0\}=\sigma(BA)\cup\{0\}$.
\end{enumerate}
\end{lemma}
\begin{proof}
(\ref{enu:AB-1}) If $ABv=\lambda v$ for some $\lambda$, then $BA(Bv)=B(ABv)=B(\lambda v)=\lambda(Bv)$,
but we require $Bv\neq0$ since eigenvectors need to be non-zero.

(\ref{enu:AB-2}) It can be seen by direct computation that if $\lambda\neq0$
and $(\lambda I-AB)$ is invertible, then $(\lambda I-BA)$ is invertible
with inverse $\frac{1}{\lambda}(I+B(\lambda I-AB)^{-1}A)$. Clearly
this argument also works if $A$ and $B$ are reversed, so $(\lambda I-AB)$ is
invertible if and only if $(\lambda I-BA)$ is. Since we required $\lambda\neq0$, we have
$\sigma(AB)\cup\{0\}=\sigma(BA)\cup\{0\}$. 
\end{proof}
Although part (\ref{enu:AB-1}) of Lemma \ref{AB} requires that $Bv\neq0$,
we will only be using analytic Toeplitz operators and composition
operators with trivial kernels when we apply the lemma.
\begin{lemma}
\label{opconv} Suppose $\ph:\D\rightarrow\D$ is analytic with Denjoy-Wolff
point $a$, $\ph_{n}\rightarrow a$ uniformly in \D, and $\psi\in \Hi$
is continuous at $z=a$. Then $\|T_{\psi(a)}\C-T_{\psi\circ\ph_{n}}\C\|\rightarrow0$
in $\mathcal{B}(H^{2})$ as $n\rightarrow\infty$. \end{lemma}
\begin{proof}
If $\ph_{n}\rightarrow a$ uniformly in \D, and $\psi$ is continuous
at $a$, then $\psi\circ\ph_{n}\rightarrow\psi(a)$ uniformly in
$D$, which implies that $\|\psi(a)-\psi\circ\ph_{n}\|_{\infty}\rightarrow0$
as $n\rightarrow\infty$. Then 
\begin{eqnarray*}
\left\Vert T_{\psi(a)}\C-T_{\psi\circ\ph_{n}}\C\right\Vert  & \leq & \left\Vert T_{\psi(a)-\psi\circ\ph_{n}}\right\Vert \left\Vert \C\right\Vert \\
 & = & \left\Vert \psi(a)-\psi\circ\ph_{n}\right\Vert _{\infty}\left\Vert \C \right\Vert \rightarrow0
\end{eqnarray*}
as $n\rightarrow\infty$, since \C\ is bounded. \end{proof}
\begin{thm}
\label{unisubset} If $\ph:\D\rightarrow\D$ is analytic with Denjoy-Wolff
point $a$, $\ph_{n}\rightarrow a$ uniformly in \D, and $\psi\in \Hi$
is continuous at $z=a$, then $\sigma(T_{\psi}\C)\subseteq\sigma(\psi(a)\C)$. \end{thm}
\begin{proof}
Note that $(T_{\psi}\C-\lambda I)$ is invertible if and only
if $(\C T_{\psi}-\lambda I)=(T_{\psi\circ\varphi}\C-\lambda I)$
is invertible by Lemma \ref{AB}. Applying this iteratively, we see
that $(T_{\psi}\C-\lambda I)$ if and only if $(T_{\psi\circ\varphi_{n}}\C-\lambda I)$
is invertible for all $n$.

Let $\lambda\in\sigma(T_{\psi}\C)$. Then $\lambda\in\sigma(T_{\psi\circ\ph_{n}}\C)$
for all $n$ by above. By Lemma \ref{opconv}, the operators $(T_{\psi\circ\ph_{n}}\C-\lambda I)$
converge to $(T_{\psi(a)}\C-\lambda I)$ in $H^{2}$ norm. Since
the invertible operators in $\mathcal{B}(H^{2})$ are an open set
and each operator in the sequence is not invertible, we know that
$(T_{\psi(a)}\C-\lambda I)$ is also not invertible, so $\lambda\in\sigma(\psi(a)\C)$. 
\end{proof}
Given the theorem above, it is seen that we assume $\psi(a)\neq0$
simply to avoid trivial cases where $\sigma(\W)=\{0\}$. Our next
goal is to find a lower bound on the spectra of \W\ and use a squeeze-type
argument. The following theorems will accomplish that.
\begin{thm}
\label{AP}Suppose $\ph:\D \rightarrow\D$ is analytic
with Denjoy-Wolff point $a$, $\ph_{n}\rightarrow a$ uniformly in
\D, and $\psi\in \Hi$ is continuous at $z=a$ with $\psi(a)\neq0$.
Then $\psi(a)$ is in $\sigma_{ap}(T_{\psi}\C)$.\end{thm}
\begin{proof}
Let 
\[
h_{m}(z)=\prod_{n=0}^{m}\frac{\psi\circ\ph_{n}}{\psi(a)}
\]
Note that $T_{\psi}\C h_{m}(z)=\psi(a)h_{m+1}(z)$. These vectors
are finite products of \Hi\ functions, so they belong in
\Hi\ and therefore to $H^{2}$ as well. The unit vectors
we will use will be $H_{m}=\frac{h_{m}}{\left\Vert h_{m}\right\Vert _{2}}$.
We wish to show $\left\Vert \left(T_{\psi}\C-\psi(a)I\right)H_{m}\right\Vert _{2}\rightarrow0$
as $m\rightarrow\infty$.
\begin{align*}
\left\Vert \left(T_{\psi}\C-\psi(a)I\right)H_{m}\right\Vert _{2} & =\frac{1}{\left\Vert h_{m}\right\Vert _{2}}\left\Vert \left(T_{\psi}\C-\psi(a)I\right)h_{m}\right\Vert _{2}\\
 & =\frac{1}{\left\Vert h_{m}\right\Vert _{2}}\left\Vert \psi(a)h_{m+1}-\psi(a)h_{m}\right\Vert _{2}\\
 & =\frac{1}{\left\Vert h_{m}\right\Vert _{2}}\left\Vert h_{m}\psi\circ\ph_{m+1}(z)-\psi(a)h_{m}\right\Vert _{2}\\
 & =\frac{1}{\left\Vert h_{m}\right\Vert _{2}}\left\Vert h_{m}\left(\psi\circ\ph_{m+1}(z)-\psi(a)\right)\right\Vert _{2}\\
 & \leq\frac{1}{\left\Vert h_{m}\right\Vert _{2}}\left\Vert h_{m}\right\Vert _{2}\left\Vert \psi\circ\ph_{m+1}(z)-\psi(a)\right\Vert _{\infty}\\
 & =\left\Vert \psi\circ\ph_{m+1}(z)-\psi(a)\right\Vert _{\infty}\rightarrow0;
\end{align*}
and the last line is by UCI holding for \ph\ and continuity
of $\psi$.\end{proof}
\begin{thm}
\label{AP2} Suppose $\ph:\D\rightarrow\D$ is analytic
with Denjoy-Wolff point $a$, $\ph_{n}\rightarrow a$ uniformly in
$D$, and $\psi\in \Hi$ is continuous at $z=a$ with $\psi(a)\neq0$.
Then for any eigenvalue $\lambda$ of \C, $\psi(a)\lambda\in\sigma_{ap}(T_{\psi}\C)$.
In particular, $\sigma_{p}(\psi(a)\C)\subseteq\sigma_{ap}(T_{\psi}\C)$.\end{thm}
\begin{proof}
Let $h_{m}$ be defined as in Theorem \ref{AP}, and let $g$ be an
eigenvector for \C\ with eigenvalue $\lambda$. Then since
the vectors $h_{m}$ are all in \Hi, $g_{m}=gh_{m}$ are
all in $H^{2}$. As before, we have 
\begin{align*}
T_{\psi}\C g_{m} & =T_{\psi}\C h_{m}g\\
 & =\left(\psi h_{m}\circ\ph\right)\left(g\circ\ph\right)\\
 & =\left(\psi(a)h_{m+1}\right)\left(\lambda g\right)\\
 & =\psi(a)\lambda h_{m+1}g\\
 & =\psi(a)\lambda g_{m+1}.
\end{align*}
Let $G_{m}=\frac{g_{m}}{\left\Vert g_{m}\right\Vert _{2}}$. Then
we have 
\begin{align*}
\left\Vert \left(T_{\psi}\C-\psi(a)\lambda I\right)G_{m}\right\Vert _{2} & =\frac{1}{\left\Vert g_{m}\right\Vert _{2}}\left\Vert (T_{\psi}\C -\psi(a)\lambda I)g_{m}\right\Vert _{2}\\
 & =\frac{1}{\left\Vert g_{m}\right\Vert _{2}}\left\Vert T_{\psi}\C g_{m}-\psi(a)\lambda g_{m}\right\Vert _{2}\\
 & =\frac{1}{\left\Vert g_{m}\right\Vert _{2}}\left\Vert \psi(a)\lambda g_{m+1}-\psi(a)\lambda g_{m}\right\Vert _{2}\\
 & =\frac{1}{\left\Vert g_{m}\right\Vert _{2}}\left\Vert \psi(a)\lambda g_{m}\left(\frac{\psi\circ\ph_{m+1}}{\psi(a)}\right)-\psi(a)\lambda g_{m}\right\Vert _{2}\\
 & =\frac{1}{\left\Vert g_{m}\right\Vert _{2}}\left\Vert \lambda g_{m}\left(\psi\circ\ph_{m+1}-\psi(a)\right)\right\Vert _{2}\\
 & \leq\frac{1}{\left\Vert g_{m}\right\Vert _{2}}\left\Vert \lambda g_{m}\right\Vert _{2}\left\Vert \psi\circ\ph_{m+1}-\psi(a)\right\Vert _{\infty}\\
 & =|\lambda|\left\Vert \psi\circ\ph_{m+1}-\psi(a)\right\Vert _{\infty}\rightarrow0.
\end{align*}
The last line is again by UCI holding for \ph\ and continuity
of $\psi$. Since this is true for any eigenvalue $\lambda$ of \C,
we have $\sigma_{p}(\psi(a)\C)\subseteq\sigma_{ap}(T_{\psi}\C)$. 
\end{proof}
Taking this together with Theorem \ref{unisubset}, we get the following
string of inequalities:
\begin{cor}
\label{squeeze}Suppose $\ph:\D\rightarrow\D$ is
analytic with Denjoy-Wolff point $a$, $\ph_{n}\rightarrow a$ uniformly
in \D, and $\psi\in \Hi$ is continuous at $z=a$ with $\psi(a)\neq0$.
Then we have 
\[
\overline{\sigma_{p}(\psi(a)\C)}\subseteq\overline{\sigma_{ap}(T_{\psi}\C)}\subseteq\sigma(T_{\psi}\C)\subseteq\sigma(\psi(a)\C)
\]

\end{cor}
In particular, if $\overline{\sigma_{p}(\C)}=\sigma(\C)$,
then $\sigma(T_{\psi}\C)=\sigma(\psi(a)\C)$.
\begin{proof}
The first containment is given by Theorem \ref{AP2}; the second containment
is trivial; the third containment is given by Theorem \ref{unisubset}. 
\end{proof}
As a consequence of this corollary, we can define the spectrum in
the case where $\ph'(a)<1$, and give some examples where $\ph'(a)=1$.
\begin{cor}
Suppose $\ph:\D\rightarrow\D$ is analytic with Denjoy-Wolff
point $a$, $\ph_{n}\rightarrow a$ uniformly in \D, and $\ph'(a)<1$.
Then for any $\psi\in \Hi$ continuous at $z=a$ with $\psi(a)\neq0$,
$\sigma(\W)=\sigma(\psi(a)\C)$.\end{cor}
\begin{proof}
When the Denjoy-Wolff point of \ph\ is on the boundary with $\ph'(a)<1$
and $a$ is the only fixed point of $\ph$, then every point in the
spectrum except for $0$ and the peripheral spectrum is an eigenvalue
of infinite multiplicity~\cite{book}. Thus $\overline{\sigma_{p}(\C)}=\sigma(\C)$
and $\sigma(\W)=\sigma(\psi(a)\C)$ by Corollary \ref{squeeze}. \end{proof}
\begin{example}
If $\ph(z)=\frac{1}{2-z}$ so that $\ph(1)=1,\ph'(1)=1$, then
it is known that \C\ has spectrum $[0,1]$ and point spectrum
$(0,1)$~\cite{book}. Since $\overline{\sigma_{p}(\C)}=\sigma(\C)$,
we have $\sigma(\W)=\sigma(\psi(a)\C)$ by Corollary~\ref{squeeze},
for any $\psi\in \Hi$ continuous at $z=1$ with $\psi(1)\neq0$.
\end{example}
So far, our work in this section has not taken the value of $\ph'(a)$
into account until the corollary above. When $\ph'(a)<1$, we can
actually be much more specific about the point spectrum, which we
will do in the next section.

\section{Point Spectra of \W\ when $\ph'(a)<1$}

For this section, our goal is to show that except for $0$ and the
peripheral spectrum, $\sigma(\W)$ otherwise consists entirely of
eigenvalues when $\ph'(a)<1$. We accomplish this by extending the
vector in the proof of Theorem \ref{AP} to an infinite series bounded
by $\ph'(a)^{n}$.
\begin{thm}
\label{eigenvector} Suppose $\ph:\D\rightarrow\D$ is analytic with
Denjoy-Wolff point $a$, $0<|\ph'(a)|<1$, and $\ph_{n}\rightarrow a$
uniformly in \D. Then for any $\psi$ in \Hi\ that is bounded
away from zero on \D\ and continuous at $a$,
there is an eigenvector $h$ for $T_{\psi}\C$ with eigenvalue
$\psi(a)$ and $h$ invertible in \Hi. \end{thm}
\begin{proof}
Since $\psi$ is a bounded, analytic, and non-vanishing map on \D,
we may assume that there exists a bounded analytic map $\eta$ so
that $\psi=e^{\eta}$. Since $\eta$ is analytic and bounded on \D,
it has bounded derivative there, so $\eta$ is Lipschitz on \D,
i.e. $|\eta(z_{1})-\eta(z_{2})|\leq K|z_{1}-z_{2}|$ for $z_{1},z_{2}\in\D$
and some constant $\tilde{K}$ independent of $z_{1},z_{2}$. Since
$\eta$ is continuous at $a$, it can be seen that the above inequality
holds on $\D\cup\{a\}$. Additionally, since $\ph_{n}$ converges
uniformly on \D, $|\ph_{n}(z)-a|\leq K\ph'(a)^n$ for
some constant $K$ independent of $z$, as seen in the proof of Theorem
\ref{julia} above. Since 
\[
\lim_{n\rightarrow\infty}\eta\circ\ph_n=\eta(a),
\]
we want to show that $\sum_{n=0}^{\infty}(\eta\circ\ph_n-\eta(a))$
converges in \Hi. Since 
\[
|\eta(\ph_n(z))-\eta(a)|\leq\tilde{K}|\ph_n(z)-a|\leq K\tilde{K}\ph'(a)^{n}
\]
and $|\ph'(a)|<1$, the series converges. Set 
\[
g=\sum_{n=0}^{\infty}(\eta\circ\varphi_{n}-\eta(a)).
\]
Then $h(z)=e^{g(z)}$ is an eigenfunction for \W\ with eigenvalue
$\psi(a)$. Since $\psi$ is bounded below, so is $\eta$, and now
$g(z)$ is bounded above and below, so $\frac{1}{h}=e^{-g(z)}$ is
also in \Hi.
\end{proof}
The next theorem shows that the special eigenvector above completely
identifies the point spectrum with that of $\psi(a)\C$.
\begin{thm}
\label{evexists} If $T_{\psi}\C$ has eigenvalue $\alpha$
with an eigenvector $g\in \Hi$ for $\alpha$, and $\lambda$
is any eigenvalue of \C\ with eigenvector $f$, then $\alpha\lambda$
is an eigenvalue of $T_{\psi}\C$ with eigenvector $gf$. Furthermore,
if $\frac{1}{g}\in \Hi$ as well, then $\sigma_{p}(\alpha \C)=\sigma_{p}(T_{\psi}\C)$.
In particular, if $\alpha=\psi(a)$, then $\sigma_{p}(\psi(a)\C)=\sigma_{p}(T_{\psi}\C)$. \end{thm}
\begin{proof}
We have $T_{\psi}\C g=\alpha g$ and $\C f=\lambda f$.
Note since $g\in \Hi,gf\in H^{2}$. Then 
\[
T_{\psi}\C(gf)=\psi(gf)\circ\ph=(\psi g\circ\ph)(f\circ\ph)=(\alpha g)(\lambda f)=\alpha\lambda(gf)
\]
so $gf$ is an eigenfunction for $T_{\psi}\C$ with eigenvalue
$\alpha\lambda$. So $\sigma_{p}(\alpha \C)\subseteq\sigma_{p}(T_{\psi}\C)$.

Now, if $\frac{1}{g}\in \Hi$ as well, then for any eigenvalue
$\mu\in\sigma_{p}(T_{\psi}\C)$ with eigenvector $h$, we can
write $v=\frac{h}{g}$ which is in $H^{2}$, so $gv=h$. Then 
\[
\mu gv=\mu h=T_{\psi}\C h=T_{\psi}\C gv=(\psi g\circ\ph)(v\circ\ph)=\alpha gv\circ\ph
\]
Dividing the far sides by $g$, we see that $\mu v=\alpha v\circ\ph$,
so $\mu\in\sigma_{p}(\alpha \C)$. Thus $\sigma_{p}(T_{\psi}\C)\subseteq\sigma_{p}(\alpha \C)$,
so now $\sigma_{p}(\alpha \C)=\sigma_{p}(T_{\psi}\C)$.
\end{proof}
Putting these theorems together, we have the following corollary.
\begin{cor}
\label{tempresult} Suppose $\ph:\D\rightarrow\D$ is analytic with
Denjoy-Wolff point $a$, $0<|\ph'(a)|<1$, and $\ph_{n}\rightarrow a$
uniformly in \D. Then for any $\psi\in \Hi$ that is bounded
away from zero on \D\ and continuous at $a$,
$\sigma_{p}(\W)=\sigma_{p}(\psi(a)\C)$. 
\end{cor}
Although we required stricter conditions on $\psi$ to achieve the
above corollary, we can in fact use UCI holding for $\ph$
to relax those conditions:
\begin{cor}
\label{pointspectrum} Suppose $\ph:\D\rightarrow\D$ is analytic
with Denjoy-Wolff point $a,0<|\ph'(a)|<1$, and $\ph_{n}\rightarrow a$
uniformly in \D. Then for any $\psi\in \Hi$ that is continuous
at $a$ with $\psi(a)\neq0$, $\sigma_{p}(\W)=\sigma_{p}(\psi(a)\C)$. \end{cor}
\begin{proof}
By Lemma \ref{AB}, $\sigma_{p}(T_{\psi}\C)=\sigma_{p}(T_{\psi\circ\ph_{n}}\C)$
for all $n$. Since $\psi$ is continuous at $a$ and $\psi(a)\neq0$,
there is $\epsilon>0$ so that $\psi(z)$ is bounded away from zero
on the set $\{z:|z-a|<\epsilon\}$. Since $\ph_{n}\rightarrow a$
uniformly, there is $N$ such that for $n\geq N$, $|\ph_{n}(z)-a|<\epsilon$,
for all $z\in\D$. Then $T_{\psi\circ\ph_{N}}\C$ satisfies
the conditions of Corollary \ref{tempresult}, so $\sigma_{p}(T_{\psi}\C)=\sigma_{p}(T_{\psi\circ\ph_{N}}\C)=\sigma_{p}(\psi(a)\C)$. 
\end{proof}
Since we can now entirely characterize the spectrum and point spectrum
when $\ph'(a)<1$, we illustrate this with an example below.
\begin{example}
Let $\ph(z)=\frac{1}{2}z+\frac{1}{2}$ and $\psi(z)=e^{(2-z)}$.
Note $\ph_{n}(z)=\frac{1}{2^{n}}z+1-\frac{1}{2^{n}}$. Then for $\eta$ as in the
proof of Theorem~\ref{eigenvector}, we have
$\eta(z)=2-z$ and we can compute $g(z)$ as in Theorem \ref{eigenvector}:
\end{example}
\begin{align*}
\sum_{n=0}^{\infty}\eta\circ\ph_{n}-\eta(1) & =\sum_{n=0}^{\infty}2-(\frac{1}{2^{n}}z+1-\frac{1}{2^{n}})-1\\
 & =\sum_{n=0}^{\infty}\frac{1}{2^{n}}(1-z)\\
 & =2-2z
\end{align*}

Then $h(z)=e^{(2-2z)}$ is an \Hi\ eigenvector for \W\
with eigenvalue $\psi(1)=e$, as is seen below: 
\[
\psi h\circ\ph=e^{(2-z)}e^{(2-2(\frac{1}{2}z+\frac{1}{2}))}=e^{(2-z)}e^{(1-z)}=e^{(1+2-2z)}=e^{1}e^{(2-2z)}=eh
\]

Note that $\frac{1}{h}=e^{(2z-2)}$ is also in \Hi. It is
known that the functions $(1-z)^{\lambda}$ are eigenvectors
of \C\ with eigenvalue $(\frac{1}{2})^{\lambda}$, that these
belong to $H^{2}$ when Re$(\lambda)>-\frac{1}{2}$, and that $\sigma_{p}(\C)=\{\lambda:0<\lambda<\sqrt{2}\}$~\cite{book}. Then $\sigma_{p}(\W)=\{\lambda:0<\lambda<\sqrt{2}e\}$
and $e^{(2-2z)}(1-z)^{\lambda}$ is an eigenvector for eigenvalue
$\frac{e}{2^{\lambda}}$. 

Our work here depended on the fact that $\ph'(a)<1$. The following
two examples show that an analogous statement cannot be made when
$\ph'(a)=1$.
\begin{example}
Let $\ph(z)=\frac{1}{2-z}$ and $\psi(z)=2-z$. Then we see that
$h(z)=1-z$ is an eigenvector for \W\ with eigenvalue $1$. It is
known that \C\ has spectrum $[0,1]$ with point spectrum $(0,1)$~\cite{book}. Since $h$ is in \Hi, any eigenvector
$g$ for an eigenvalue $\lambda$ of \C\ corresponds to an
eigenvector $gh$ of \W\ with eigenvalue $\lambda$. Thus \W\
has spectrum $[0,1]$ and every element $0<\lambda\leq1$ is an eigenvalue. 
\end{example}

\begin{example}
Let $\ph(z)=\frac{1}{2-z}$ and $\psi(z)=\frac{1}{1-\frac{1}{2}z}$.
The first two authors~\cite{wcomp} showed that the operator
\W\ is self-adjoint and has no eigenvalues, but rather consists
entirely of approximate point spectrum. 
\end{example}

\section{Seminormality of \W}

In~\cite{wcomp}, the first two authors showed that the semigroup
of parabolic non-automorphisms studied in this paper have a companion
weight so that \W\ is self-adjoint. The form of the companion weight
associated with most known self-adjoint~\cite{wcomp}, normal~\cite{BN}, and 
cohyponormal~\cite{CJK}
weighted composition operators is $\psi=pK_{\sigma(0)}$, where $p$
is a constant and $\sigma$ is the Cowen auxillary function of $\ph$
(which is linear fractional in these situations). As a result of our
work above, we eliminate possibilities for $\psi$ when \ph\ is
a linear fractional non-automorphism with Denjoy-Wolff point on $\partial D$
and \W\ is seminormal.

First, we show that when \ph\ is a parabolic non-automorphism,
there are no other weight functions $\psi$ continuous at the Denjoy-Wolff
point $a$ so that \W\ is (co)hyponormal.
\begin{thm}
Let $\ph:\D\rightarrow\D$ be a parabolic non-automorphism with Denjoy-Wolf
point $a$ and let $\psi\in \Hi$ be continuous at $z=a$.
If \W\ is (co)hyponormal, then it is normal and $\psi$ is a multiple
of $K_{\sigma(0)}$, where $\sigma$ is the Cowen auxillary function
of $\ph$. Furthermore, if $\psi(a)$ is real, then \W\ is self-adjoint. \end{thm}
\begin{proof}
Without loss of generality, assume $a=1$ since composition with a
rotation is unitary. For now, assume $\psi(a)$ is real. Any (co)hyponormal
operator whose spectrum has zero area is normal~\cite{MP}.
Since \C\ has spectrum $[0,1]$ and point spectrum $(0,1)$,
\W\ (and therefore also \Ws) has spectrum equal to the line
segment $[0,\psi(a)]$ by Corollary \ref{squeeze}. Since line segments
have zero area, \W\ is normal. Since \W\ is (sub)normal and $\sigma(\W)\subseteq\mathbb{R}$,
it is self-adjoint~\cite{conway}. The self-adjoint weighted
composition operators on $H^{2}$ have been completely characterized
in~\cite{wcomp} and $\psi$ must therefore be a real multiple
of $K_{\sigma(0)}$.

If $\psi(a)$ is not real, we get the same result for the weight $\lambda\psi$,
where $\lambda$ is a non-zero constant so that $\lambda\psi(a)$
is real. Then we see that $\psi$ must be a (non-real) multiple of
$K_{\sigma(0)}$ and that \W\ is normal. 
\end{proof}
Next, let \ph\ be a hyperbolic non-automorphism. Here, \W\ can
be cohyponormal and in fact cosubnormal. For example, if $\ph(z)=sz+1-s,0<s<1$,
then $\sigma(0)=0,K_{0}=1$ and \C\ is a ``weighted'' composition
operator which is cosubnormal. In~\cite{BN}, it
is shown that if $\psi$ is in $C^{1}$ on $\overline{\D}$ then \W\
cannot be essentially normal. Due to our understanding of the spectrum
from Section 4 above, we can show that no weight $\psi$ in \Hi\
continuous at the Denjoy-Wolff point (but with no other conditions
on $\psi$ at the boundary) creates a hyponormal weighted composition
operator when \ph\ is a hyperbolic non-automorphism. However, first
we need a lemma.
\begin{lemma}
\label{inprod} Let $g$ be a vector in $H^{2}$ such that $\left<g,z^{n}g\right>=\left<g,g\right>$
for all integers $n\geq1$. Then $g$ is the zero vector. \end{lemma}
\begin{proof}
Suppose that $g$ is not the zero vector. Writing $g$ as $g=\sum_{k=0}^{\infty}a_{k}z^{k}$,
since $g$ is not the zero vector, not all $a_{k}$ are zero. Therefore,
there is an integer $n$ such that the vector $g_{n}=\sum_{k=n}^{\infty}a_{k}z^{k}$
satisfies $\left\Vert g_{n}\right\Vert <\left\Vert g\right\Vert /2$.
Then 
\[
\left\Vert g\right\Vert ^{2}=\left|\left<g,g\right>\right|=\left|\left<g,z^{n}g\right>\right|=\left|\left<g_{n},g\right>\right|\leq\left\Vert g_{n}\right\Vert \left\Vert g\right\Vert <\left\Vert g\right\Vert ^{2}/2
\]
which is impossible. Therefore $g$ is the zero vector.\end{proof}
\begin{thm}
Let $\ph:\D\rightarrow\D$ be a hyperbolic non-automorphism. There
is no $\psi\in \Hi$ continuous at $z=a$ such that \W\ is
hyponormal. \end{thm}
\begin{proof}
Without loss of generality, assume $\ph(z)=sz+1-s$ for some $0<s<1$.
(Otherwise, conjugate \W\ by the unitary weighted composition operator
$T_{g}C_{\zeta}$ where $g=K_{\zeta(0)}$ and $\zeta$ is an automorphism
so that $\zeta\circ\ph\circ\zeta$ is in this form. This will change
the weight function $\psi$, but it will still be continuous at $a$
and it is otherwise arbitrary.) Now assume \W\ is hyponormal.

It is known that $(1-z)^{n}$ is an eigenvector for \C\ with
eigenvalue $s^{n}$. By Theorem \ref{eigenvector}, there is an eigenfunction
$h\in \Hi$ for \W\ with eigenvalue $\psi(a)$, and thus
$h(1-z)^{n}$ is an eigenfunction for \W\ with eigenvalue $\psi(a)s^{n}$
by Theorem~\ref{evexists}. Since \W\ is hyponormal, eigenvectors
corresponding to different eigenvalues must be perpendicular~\cite{conway}.
Then 
\[
0=\left<h,(1-z)h\right>=\left<h,h\right>-\left<h,zh\right>\Rightarrow\left<h,h\right>=\left<h,zh\right>
\]
Keeping this result in mind, we now consider the vectors $h$ and
$(1-z)^{2}h$: 
\[
0=\left<h,(1-z)^{2}h\right>=\left<h,h\right>-2\left<h,zh\right>+\left<h,z^{2}h\right>\Rightarrow\left<h,h\right>=\left<h,z^{2}h\right>
\]

Continuing inductively, we have $\left<h,h\right>=\left<h,z^{n}h\right>$
for all integers $n>0$. Therefore, by Lemma \ref{inprod}, $h$ is
the $0$ vector, which is a contradiction since eigenvectors are non-zero.
Therefore \W\ cannot be hyponormal. 
\end{proof}

\section{Further Questions}

Below is a list of questions that would extend our work:
\begin{enumerate}
\item Characterize exactly when the iterates $\ph_n$ converge uniformly  to $a$ on all
of \D.
\item Completely characterize the point spectrum of \W\ when $|a|=1$,  $\ph'(a)=1$
and the iterates $\ph_n$  converge uniformly to $a$ in all of \D.
\item Completely characterize (co)(hypo)normal weighted composition operators
on $H^{2}$. (For example, it has \textit{not} been shown that if
\W\ is normal, \ph\ must be linear fractional.)
\item In our work and many of our referenced papers, it seems that when
\ph\ has exactly one fixed point $a$ in $\overline{\D}$, that
$\sigma(\W)=\sigma(\psi(a)\C)$. How often is this true?
\end{enumerate}


\vspace{.3in}



\begin{thebibliography}{99}
\normalsize
\bibitem{BMS} { P.~S.~Bourdon, V.~Matache, {\rm and} J.~H.~Shapiro},
{On convergence to the Denjoy-Wolff point}, \ill{49}{2005}{405}{430}

\bibitem{BN} { P.~S.~Bourdon {\rm and} S.~Narayan},
{Normal weighted composition operators on the Hardy space $\Ht(U)$}, 
\jmaa{367}{2010}{278}{286}

\bibitem{conway} {J.~B.~Conway,}
{\em The Theory of Subnormal Operators}, Math. Surveys and Monographs, vol. 36,
American Mathematical Society, Providence, RI, 2000.

\bibitem{commut} { C.~C.~Cowen},
{The commutant of an analytic Toeplitz operator}, \tams{239}{1978}{1}{31}

\bibitem{commut2} { C.~C.~Cowen,}
{An analytic Toeplitz operator that  commutes with a compact operator}, \jfa{ 36}{1980}{169}{184}
  
\bibitem{Co88}{ C.~C.~Cowen,} {Linear fractional
composition operators on \Ht}, \inteq{11}{1988}{151}{160}
 
\bibitem{book} {C.~C.~Cowen {\rm and }B.~D.~MacCluer,}
{\em Composition Operators on Spaces of Analytic Functions,} CRC
Press, Boca Raton, 1995.
 
\bibitem{adj} { C.~C.~Cowen {\rm and }E.~A.~Gallardo Gutierrez,}
{A new class of operators and a description of adjoints of composition operators},
\jfa{238}{2006}{447}{462}

\bibitem{wcomp}{\caps C.~C.~Cowen {\rm and }E. Ko,} {Hermitian weighted
composition operators on \Ht},\\ \tams{362}{2010}{5771}{5801}

\bibitem{CJK}{\caps C.~C.~Cowen, S.~Jung, {\rm and }E. Ko,} {Normal and cohyponormal weighted composition operators on \Ht}, to appear: {\jour Operator Theory: Advances and Applications}.

\bibitem{forelli}{ F.~Forelli,} {The isometries of $H^p$},  \canj{16}{1964}{721}{728}

\bibitem{gunatTh}{G.~Gunatillake,} {\em Weighted Composition 
Operators,} Thesis, Purdue University, 2005.

\bibitem{gunatInv}{G.~Gunatillake,} {Invertible weighted composition 
operators,} \jfa{261}{2011}{831}{860}

\bibitem{HLNS}{O.~Hyv\"arinen, M. Lindstr\"om, I. Nieminen, {\rm and }E. Saukko,} { 
Spectra of weighted composition operators with automorphic symbols,} \jfa{261}{2011}{831}{860}

\bibitem{MP} {M.~Martin {\rm and }M.~Putinar,}
{\em Lectures on Hyponormal Operators,} Operator Theory: Advances and Applications, vol. 39.
Birkhauser Verlag, Basel, 1989.

\end{thebibliography}
\end{document}